\documentclass[11pt]{article}
\usepackage{amsmath,amssymb,amsthm}
\newcommand{\re}{\mathbb{R}}
\newcommand{\co}{\mathbb{C}}

\newcommand{\T}{\mathrm{T}}
\newcommand{\D}{\mathrm{D}}
\newcommand{\R}{\mathrm{R}}
\newcommand{\rp}{\mbox{Re}}

\long\def\symbolfootnote[#1]#2{\begingroup%
\def\thefootnote{\fnsymbol{footnote}}\footnote[#1]{#2}\endgroup}

\newtheorem{thm}{Theorem}[section]

\newtheorem{lem}[thm]{Lemma}
\newtheorem{cor}[thm]{Corollary}

\theoremstyle{definition}

\theoremstyle{remark}

\usepackage{comment}

\title{On local existence of solutions for nonlinear systems of Cauchy-Riemann operator of any order in the plane}

\author{Yifei Pan}

\begin{document}

\maketitle
\begin{center}

\end{center}
\begin{abstract}
We prove a general local existence theorem for nonlinear systems of Cauchy-Riemann operator of any order in one complex variable with initial values at a given point, which is a counterpart of local existence of ODE.

\end{abstract}

\large

\section{Introduction}\label{sec0}\symbolfootnote[0]{MSC 2010: 35G20 (Primary); 32G05, 30G20 (Secondary)}

In this paper we consider local existence of solutions for nonlinear partial differential systems of Cauchy-Riemann operator in dimension two from point view of complex analysis.

First we introduce some simple notations. Let $D$ denote the closed disk $\{z\in\co\mid|z|\leq R\}$. Let $m$ be a positive integer. For each $k$: $1\leq k\leq m$, we denote $\mathcal{D}^k v$
as a vector in $\co^{2^k}$ with entries $\partial^\mu\bar\partial^\nu v$ where $\mu+\nu=k$ where $v$ is a complex valued function on $D$. Here we use the standard complex derivatives for partial derivatives as
$\partial=\frac{1}{2}(\frac{\partial}{\partial x}-i\frac{\partial}{\partial y}), \mbox{    }\bar\partial=\frac{1}{2}(\frac{\partial}{\partial x}+i\frac{\partial}{\partial y}),$
where $z=x+iy$.
Let $u$ be a map from $D$ to $\co^n$ with $u=(u^1,..., u^n)$. We denote $\mathcal{D}^k u$ as $(\mathcal{D}^k u^1,..., \mathcal{D}^k u^n)$, which is a vector in $\co^{2^kn}$.

Let $\Omega$ be the set
$$\Omega=D\times \co^n\times\co^{2n}\times\cdot\cdot\cdot\times\co^{2^{m-1}n},$$
with coordinates as $(z, \eta_0, \eta_1, ..., \eta_{m-1})$.

The following is the main result of this paper, which can be considered as an analogue for systems of ordinary differential equations with initial conditions.

\begin{thm}
 Let $a:\Omega \to \co^n$ be any function of class $C^{1+\alpha}$. Let $\mu, \nu$ be nonnegative integers such that $\mu+\nu=m\geq 1$, and $c_{i,\bar j}$ be any vector in $\co^n$ for $i+j\leq m-1$. Then the system of Cauchy-Riemann operators with initial values at the origin:
\begin{eqnarray}
\partial^\mu\bar\partial^\nu u(z)&=&a(z, u,\mathcal{D}^1 u, ..., \mathcal{D}^{m-1}u)\nonumber\\
\partial^i\bar\partial^ju(0)&=&c_{i,\bar j}, \mbox{ for } i+j\leq m-1\nonumber
\end{eqnarray}
has (infinitely many) solutions $u:D\to \co^n$ of class $C^{m+\alpha}$ in $D$ for sufficiently small values of $R$.

\end{thm}
As a special case we state the following
\begin{cor}
Let $a:D\times \co^n\to \co^n$ be any function of class $C^{1+\alpha}$. Let $c$ be any vector in $\co^n$.
Then the following nonlinear first odrer partial differential system of Cauchy-Riemnn operator
\begin{eqnarray}
\bar\partial u(z)&=&a(z, u(z))\nonumber\\
u(0)&=&c\nonumber
\end{eqnarray}
has (infinitely many) solutions $u:D\to \co^n$ of class $C^{1+\alpha}$ in $D$ for sufficiently small values of $R$.
\end{cor}

In order to prove Theorem 1.1, it suffices to prove the case of zero initial values. The general case of initial values can be converted easily. The whole paper is devoted to give a proof for the case of zero initial values.

Finally, The main idea and method have been applied to higher dimensions of elliptic operators [PZ], {PY]. 

\subsection{Examples of no solutions-Mizohata equation}
After a famous example of Lewy [L] in $\re^3$, Mizohata [M] considered, in $\re^2$, the following equation
$$\frac{\partial u}{\partial x}+i x\frac{\partial u}{\partial y}=F(x,y)\label{eq0}.$$
It was proved in [M] that there is a smooth function $F$ for which the above equation has no solution near the origin. Converting the equation to complex one, one has the following equation
\begin{eqnarray}
\bar\partial u =\frac{1}{1+\rp z}F(z,\bar z)-\frac{1-\rp z}{1+\rp z}\partial u.
\end{eqnarray}
According to the notation as Theorem 1.1, we have
$$a(z, \eta_0, \eta)=\frac{1}{1+\rp z}F(z,\bar z)-\frac{1-\rp z}{1+\rp z}\eta,$$
whence
$$\partial_\eta(0)=-1, \bar\partial_\eta a(0)=0.$$
Therefore the condition (1) is not met for Theorem 1.1. Taking $\partial^\mu\bar\partial^{\nu-1}$ on both side, we have
$$\partial^\mu\bar\partial^{\nu}u=-\frac{1-\rp z}{1+\rp z}\partial^{\mu+1}\bar\partial^{\nu-1}u+\cdot\cdot\cdot+\partial^\mu\bar\partial^{\nu-1}\left\{\frac{1}{1+\rp z}F(z,\bar z)\right\}.$$
This is a differential equation of any order which has no solutions locally at the origin.

\subsection{Osserman's theorem}
We remark that it is classical that the equation $\Delta u=e^{2u}$ has no solutions in the whole plane $\co$ due to Ahlfors [A], and local existence puts constrain on the radius of existence due to Osserman [O]. Therefore, the vanishing condition in Theorem B can't be dropped, and the existence of radius in Theorem A can't be arbitrarily large. In fact, solve a solution $u$ by Theorem A, of
$$\Delta u=e^{2u}, u(0)=a, \nabla u(0)=b$$
for $|z|\leq R$. Then by [O], we have $R\leq 2e^{-a}$, which goes to zero if $a\to\infty$.

Also we can solve the following
$$\bar\partial u=u^2$$
specifically to get
$$u(z)=\frac{1}{\varphi(z)+\bar z}$$
where $\varphi(z)$ is an entire function. From this, we can see the radius of existence of the solution can be any radius $R$.
On the other hand, we can solve
$$\bar\partial u=e^u$$
to get
$$u(z)=\ln(\varphi(z)+\bar z)$$
where we can define a branch cut so that $u$ is well-defined.
\section{Function spaces and their norms}

\subsection{H\" older space}

Let $C^{\alpha}(D)$ be the set of all functions $f$ on $D$ for which
$$ H_{\alpha}[f]=\sup\left\{{\frac{|f(z)-f(z')|}{|z-z'|^{\alpha}} \bigg| z,z' \in \D} \right\}$$
is finite. Let $C^{k}(D)$ be the set of all function $f$ on $D$ whose $k^{\textup{th}}$ order partial derivatives exist and are continuous, $k$ an integer, $k \geq 0$.
$C^{k+\alpha}(D)$ is the set of all functions $f$ on $D$ whose $k^{\textup{th}}$ order partial derivatives exist and belong to $C^{\alpha}(D)$.

The symbol $|f|$ or $|f|_\D$ denotes $\textup{sup}_{z\in D}|f(z)|$.
For $f\in C^\alpha(D)$ we define
$$\|f\|=|f|+(2R)^\alpha H_\alpha[f].$$
The set of $n$-tuples $f=(f_1,..., f_n)$ of functions (vector functions or maps) of $C^\alpha (D)$ is denoted by $[C^\alpha (D)]^n$, and $H_\alpha[f]$ is defined as the maximum of $H_\alpha[f_i](i=1,..,n)$. In a similar fashion we define $|f|_A=\sup_{z\in A}|f(z)|$ for functions and vector functions, and write $|f|$ when the domain is understood. Finally, in this paper throughout, the norm of $\co^N$ is taken as $|v|=\max|v_j|$.

The following lemma is well-known; for a proof see ([NW], 7.1b).
\begin{lem}
The function $\|\cdot\cdot\cdot\|$ defined on $C^{\alpha}(D)$ is a norm, with respect to which $C^{\alpha}(D)$ is a Banach algebra: $\|fg\|\leq \|f\|\|g\|$.
\end{lem}
The following simple lemma is to be used multiple times throughout the paper.
\begin{lem} If $f\in C^{k+\alpha}(D)$, then
$$|f(z')-\sum_{l=0}^k\frac{1}{l!}\sum_{i+j=l}\partial^i\bar\partial^j f(z)(z'-z)^i{(\bar z'-\bar z)}^j|\leq \bigg\{\sum_{i+j=k}H_\alpha[\partial^i\bar\partial^j f]\bigg\}|z'-z|^{k+\alpha}.$$
\end{lem}
\begin{proof}
Expanding at $z$, we have the formula
\begin{eqnarray*}
&&f(z')-\sum_{l=0}^{k-1}\frac{1}{l!}\sum_{i+j=l}\partial^i\bar\partial^j f(z)(z'-z)^i{(\bar z'-\bar z)}^j\\
&=&\int_0^1\int_0^{t_{k-1}}\cdot\cdot\cdot\int_0^{t_1}\left\{\frac{d^k}{dt^k}f(t z'+(1-t)z)\right\}dtdt_1\cdot\cdot\cdot dt_{k-1}\nonumber\\
&=&\int_0^1\int_0^{t_{k-1}}\cdot\cdot\cdot\int_0^{t_1}\left\{\sum_{i+j=k}\partial^i\bar\partial^j f(tz'+(1-t)z)(z'-z)^i{(\bar z'-\bar z)}^j\right\}dtdt_1\cdot\cdot\cdot dt_{k-1}.\nonumber
\end{eqnarray*}
Hence, by subtracting kth term $\frac{1}{k!}\sum_{i+j=k}\partial^i\bar\partial^j f(z)(z'-z)^i{(\bar z'-\bar z)}^j$ from both sides, we have
$$f(z')-\sum_{l=0}^k\frac{1}{l!}\sum_{i+j=l}\partial^i\bar\partial^j f(z)(z'-z)^i{(\bar z'-\bar z)}^j$$
$$=\int_0^1\int_0^{t_{k-1}}\cdot\cdot\cdot\int_0^{t_1}\left\{\sum_{i+j=k}\{\partial^i\bar\partial^j f(tz'+(1-t)z)-\partial^i\bar\partial^j f(z)\}(z'-z)^i{(\bar z'-\bar z)}^j\right\}dtdt_1\cdot\cdot\cdot dt_{k-1}$$
Thus we have,
$$|f(z')-\sum_{l=0}^k\frac{1}{l!}\sum_{i+j=l}\partial^i\bar\partial^j f(z)(z'-z)^i{(\bar z'-\bar z)}^j|$$
$$\leq\sum_{i+j=k}H_\alpha[\partial^i\bar\partial^j f]|z'-z|^{k+\alpha}.$$
This completes the proof.
\end{proof}
\subsection{Function spaces with vanishing order at the origin}
Our idea of solving differential equations of order $m$ is to look for solutions that vanish up to $m-1$ order at the origin; this way the norm estimate of the function space to be considered later is made possible using only $m$th order derivatives. 
The general case of initial values can be converted to the special case of zero initial values. We denote for $k\geq 1$, $C_0^{k+\alpha}(D)$ the set of all functions in $C^{k+\alpha}(D)$ whose derivatives vanish up to order $k-1$ at the origin. Specifically
$$C_0^{k+\alpha}(D)=\{f\in  C^{k+\alpha}(D)\big | \partial^i\bar \partial^j f(0)=0, i+j\leq k-1\}.$$

We now define functions $\|\cdot\cdot\cdot\|^{(k)}$ on $C^{k+\alpha}(D)$ inductively. On $C^{1+\alpha}(D)$ we define, following [NW],
$\|f\|^{(1)}=\max\{\|\partial f\|, \|\bar\partial f\|\}.$
For $k\geq 2$, we define
$\|f\|^{(k)}=\max\{\|\partial f\|^{(k-1)}, \|\bar\partial f\|^{(k-1)}\}.$
Obviously, we have the definition $\|\cdot\cdot\cdot\|^{(k)}$ in terms of $\|\cdot\cdot\cdot\|$:
$\|f\|^{(k)}=\max_{i+j=k}\{\|\partial^i\bar\partial^j f\|\}.$
We point out that the function $\|\cdot\cdot\cdot\|^{(k)}$ on $C^{k+\alpha}(D)$ is not a norm since $\|f\|^{(k)}=0$ if and only if $f$ is
a polynomial of degree at most $k-1$. However it becomes norm when restricted to the subspace $C_0^{k+\alpha}(D)$, which is to be proved below. In this paper, we denote $C$ as a constant independent of $R$ varying from line to line. First we obtain some useful estimates.
\begin{lem}
If $f\in C_0^{k+\alpha}(D)$, then
$\|f\|\leq CR^k\|f\|^{(k)}.$
\end{lem}
\begin{proof}
Let $f\in C_0^{k+\alpha}(D)$, then
\begin{eqnarray*}
f(z)&=&\int_0^1\int_0^{t_{k-1}}\cdot\cdot\cdot\int_0^{t_1}\left\{\frac{d^k}{dt^k}f(t z)\right\}dtdt_1\cdot\cdot\cdot dt_{k-1}\nonumber\\
 &=&\int_0^1\int_0^{t_{k-1}}\cdot\cdot\cdot\int_0^{t_1}\left\{\sum_{i+j=k}\partial^i\bar\partial^j f(t z)z^i\bar z^j \right\}dtdt_1\cdot\cdot\cdot dt_{k-1}\nonumber\\
&=&\sum_{i+j=k}\left\{\int_0^1\int_0^{t_{k-1}}\cdot\cdot\cdot\int_0^{t_1}\partial^i\bar\partial^j f(t z) dtdt_1\cdot\cdot\cdot dt_{k-1}\right\}z^i\bar z^j.\nonumber
\end{eqnarray*}
Applying norm inequality, we obtain
\begin{eqnarray*}
\|f\|&\leq & \sum_{i+j=k}\frac{1}{k!}\|\partial^i\bar\partial^j f\|\|z^i\bar z^j\|\nonumber\\
&\leq &\sum_{i+j=k}\frac{1}{k!}\|\partial^i\bar\partial^j f\|\|z\|^k\leq CR^k\|f\|^{(k)}\nonumber,
\end{eqnarray*}
where we have used $\|z\|=3R$, which is easily verified.
\end{proof}
\begin{lem}
If $f\in C_0^{m+\alpha}(D)$, then, for $l\leq m$,
$$\|f\|^{(l)}\leq CR^{m-l}\|f\|^{(m)}.$$
\end{lem}
\begin{proof}
Let  $f\in C_0^{m+\alpha}(D)$. If $i+j=l$, then $\partial^i\bar\partial^j f\in C_0^{m-l+\alpha}(D)$. By Lemma 2.3, we have
$$\|\partial^i\bar\partial^j f\|\leq CR^{m-l}\|\partial^i\bar\partial^j f\|^{(m-l)}\leq CR^{m-l}\| f\|^{(m)}.$$
\end{proof}

\begin{lem}
The function space $C_0^{k+\alpha}(D)$ equipped with the function $\|\cdot\cdot\cdot\|^{(k)}$ is a Banach space.
\end{lem}
\begin{proof}
First define $|||f|||=\sum_{i+j=0}^k\|\partial^i\bar\partial^j f\|$. It is well known that $C^{k+\alpha}(D)$ equipped with the norm $|||\cdot|||$ is a Banach space. By Lemma 2.4 the norm $\|\cdot\cdot\cdot\|^{(k)}$ for $C_0^{k+\alpha}(D)$ is equivalent to the norm $|||\cdot|||$  of $C^{k+\alpha}(D)$. Note that $C_0^{k+\alpha}(D)$ is a closed subspace of $C^{k+\alpha}(D)$ and is therefore a Banch space with norm $\|\cdot\cdot\cdot\|^{(k)}$.
\end{proof}

\section{Cauchy-Green operator and high order derivative formula}
\subsection{Basic definitions and properties}
The operators are defined for integrable functions on $D$ and $C=\{|z|=R\}$ as follows:
\begin{eqnarray}
 Tf(z)&=&\frac{-1}{2\pi i}\int_D\frac{f(\zeta)d\bar{\zeta}\wedge d\zeta}{\zeta-z},\mbox{   }^2Tf(z)=\frac{-1}{2\pi i}\int_D\frac{f(\zeta)-f(z)}{(\zeta-z)^2}d\bar\zeta\wedge d\zeta\nonumber\\
Sf(z)&=&\frac{1}{2\pi i}\int_C\frac{f(\zeta)d\zeta}{\zeta-z},\hspace{1cm}
 S_bf(z)=\frac{1}{2\pi i}\int_C\frac{f(\zeta)d\bar\zeta}{\zeta-z},\nonumber\\
.\nonumber
\end{eqnarray}
Using polar coordinates, we see that $Tf$ is defined for continuous $f$, and $^2Tf$ for $C^\alpha(D)$. The operator $S$ is the familiar Cauchy integral.
We also define related operators: $\overline T, \overline S,\overline S_b$, and $^2\overline T$ as follows: $\overline T(f)=\overline {T(\bar f)}$, $\overline S(f)=\overline {S(\bar f)}$, $\overline S_b(f)=\overline {S_b(\bar f)}$ and $^2\overline T(f)=\overline {^2T(\bar f)}$ .
Importantly, we will use operators of $T^\mu\overline{T}^\nu$, which is a composition of $T,\overline T$. Here we assume $T^0=\mathrm{Id},\overline T^0=\mathrm{Id}$ and etc.
We note here that $S_b$ in this paper is the same as $\overline S$ in [NW]. The following estimate holds:
\begin{eqnarray}
 |Tf|\leq 4R|f|\nonumber.
\end{eqnarray}
More generally, if $\triangle$ is a bounded domain, then $T_\triangle f$ is defined for continuous $f$ on $\triangle$ by
$$T_\triangle f(z)=\frac{-1}{2\pi i}\int_\triangle\frac{f(\zeta)d\bar{\zeta}\wedge d\zeta}{\zeta-z}, \mbox{ }S_\triangle f(z)=\frac{1}{2\pi i}\int_{\partial\triangle}\frac{f(\zeta) d\zeta}{\zeta-z} .$$
We have
$$ |T_\triangle f|\leq 2 \mathrm{diam}(\triangle)|f|_\triangle.$$
The fundamental property between operators $T,S$ is the following ([NW], 6.1a).
\begin{lem}
If $f\in C^1(D)$, then
$$ T\bar\partial f=f-Sf     \mbox{         on         }\mbox{} \mathrm{Int}(D),$$
$$ \overline T\partial f=f-\overline Sf     \mbox{         on         }\mbox{} \mathrm{Int}(D).$$
\end{lem}
\begin{lem}
If $f\in C^{m+\alpha}(D)$ and $\mu+\nu=m$, then
$$\T^\nu\overline T^\mu(\partial^\mu\bar\partial^\nu f)=f-\sum_{j=0}^{\nu-1}T^j(S(\bar\partial^j f))-\sum_{j=0}^{\mu-1}T^\nu\overline
T^j(\overline S(\partial^j\bar\partial^\nu f)) \mbox{ on }\mathrm{Int}(D)$$
if $\mu, \nu\geq 1$; otherwise
$$T^m(\bar\partial^m g)=g-\sum_{j=0}^{m-1}T^j(S(\bar\partial^j g))$$
$$\overline T^m(\partial^m g)=g-\sum_{j=0}^{m-1}\overline T^j(\overline S(\partial^j g)).$$
\end{lem}
\begin{proof}
Just apply Lemma 3.1 repeatedly.
\end{proof}

\subsection{A new high order integral operator}
For $k\geq 0$, we define a new operator $^{k+2}Tf$ on $C^{k+\alpha}(D)$ by
$$^{k+2}Tf(z)=\frac{-(k+1)!}{2\pi i}\int_D\frac{f(\zeta)-P_{k}(\zeta,z)}{(\zeta-z)^{k+2}}d\bar\zeta\wedge d\zeta,$$
where $P_{k}(\zeta,z)$ is the Taylor expansion of $f$ at $z$ of degree $k$. Namely
$$P_{k}(\zeta,z)=\sum_{l=0}^k\frac{1}{l!}\sum_{i+j=l}\partial^i\bar\partial^j f(z)(\zeta-z)^i{(\bar \zeta-\bar z)}^j.$$
We also define $^{k+2}\overline Tf=\overline {^{k+2}T(\bar f)}$.
More generally, if $\Delta$ is a bounded domain, then $^{k+2}T_\Delta$ is  defined as
$$^{k+2}T_\Delta f(z)=\frac{-(k+1)!}{2\pi i}\int_\Delta\frac{f(\zeta)-P_{k}(\zeta,z)}{(\zeta-z)^{k+2}}d\bar\zeta\wedge d\zeta.$$
Note the case $k=0$ was defined in [NW]. The following is well-known and classical(see [NW], [V]).
\begin{lem}
Let $f\in C^\alpha(D)$. Then $Tf\in C^{1+\alpha}(D), ^2Tf\in C^\alpha(D)$. Moreover
$$\bar\partial Tf=f,\mbox{  }\partial Tf=^2Tf,$$
$$H_\alpha[^2Tf]\leq C H_\alpha[f],$$
If $f\in C^{k+\alpha}(D)(k\geq 0)$, then $Tf\in f\in C^{k+1+\alpha}(D)$.
\end{lem}

We prove the following extension of Lemma 3.3 for the operator $^{k+2}Tf$, which is one of key tools in the proof of the results that follow.
\begin{thm}
Let $f\in C^{k+\alpha}(D) (k\geq 0)$. Then
$$\partial^i\bar\partial^j Tf=\partial^i\bar\partial^{j-1}f,$$
if $j\geq 1, i+j\leq k+1;$ otherwise
$$\partial^{k+1} Tf=^{k+2}Tf.
$$

\end{thm}
Before giving a proof of Theorem 3.4, we need two lemma on compution of integrals.
\begin{lem}
If $\varphi(\zeta) $ is holomorphic in $\co$, then
$$\int_\triangle\frac{\overline{\phi(\zeta)}}{\zeta-w}d\bar\zeta\wedge d\zeta$$
is anti-holomorphic in $w\in \mathrm{Int}(\triangle)$, where $\triangle=\{|\zeta-z_0|\leq r\}.$
\end{lem}
\begin{proof} Writing $\phi(\zeta)=\sum a_l(\zeta-z_0)^l$, we only have to prove the lemma for $\phi(\zeta)=(\zeta-z_0)^l$. Indeed,
\begin{eqnarray}
\int_{\triangle} \frac{(\bar\zeta-\bar z_0)^l}{\zeta-w}d\bar\zeta\wedge d\zeta &=&{\frac{-2\pi i}{(l+1)}}T_{\triangle}(\bar\partial (\bar\zeta-\bar z_0)^{l+1})(w)\nonumber\\
&=&{\frac{-2\pi i}{(l+1)}}\{(\bar w-\bar z_0)^{l+1}-S_\triangle((\bar\zeta-\bar z_0)^{l+1})(w)\},\nonumber
\end{eqnarray}
where we have
\begin{eqnarray}
S_\triangle((\bar\zeta -\bar z_0)^{l+1})(w) &=& \frac{1}{2\pi i}\int_{|\zeta-z_0|=r}\frac{(\bar\zeta-\bar z_0)^{l+1}}{\zeta-w}d\zeta\nonumber\\
&=& \frac{r^{2(l+1)}}{2\pi i}\int_{|\zeta-z_0|=r}\frac{1}{(\zeta-z_0)^{l+1}(\zeta-w)}d\zeta\nonumber\\
&=&0,\nonumber
\end{eqnarray}
where in the last equality, we have used the residue theorem to get zero integral.
\end{proof}
\begin{cor}It holds for $l\geq 0$,
$$\int_{\triangle} \frac{(\bar\zeta-\bar z_0)^l}{\zeta-w}d\bar\zeta\wedge d\zeta={\frac{-2\pi i}{(l+1)}}(\bar w-\bar z_0)^{l+1}$$
where $\triangle=\{|\zeta-z_0|\leq r\}.$
\end{cor}
\begin{lem}
If $l\geq 1$, then
$$\int_\triangle \frac{(\bar\zeta-\bar z)^l}{(\zeta-z)^{l+1}}d\bar\zeta\wedge d\zeta=0$$
for $z\in \mathrm{Int(\triangle)}.$
\end{lem}
\begin{proof}
We notice that $\frac{(\bar\zeta-\bar z)^l}{(\zeta-z)^{l+1}}$ is integrable in $\triangle$ using polar coordinate at $z$. Let $\epsilon(z)$ be the disk of radius $\epsilon$ and center at $z$ so that $\epsilon(z)\subset\triangle$.
Now we have
\begin{eqnarray}
\int_\triangle\frac{(\bar\zeta-\bar z)^l}{(\zeta-z)^{l+1}}d\bar\zeta\wedge d\zeta &=&
\int_{\triangle\setminus \epsilon(z)} \frac{(\bar\zeta-\bar z)^l}{(\zeta-z)^{l+1}}d\bar\zeta\wedge d\zeta+\int_{\epsilon(z)} \frac{(\bar\zeta-\bar z)^l}{(\zeta-z)^{l+1}}d\bar\zeta\wedge d\zeta\nonumber\\
&=&
\frac{1}{l!}\frac{d^l}{dw^l}\int_{\triangle\setminus \epsilon(z)} \frac{(\bar\zeta-\bar z)^l}{\zeta-w}d\bar\zeta\wedge d\zeta\bigg |_{w=z}+\int_{\epsilon(z)} \frac{(\bar\zeta-\bar z)^l}{(\zeta-z)^{l+1}}d\bar\zeta\wedge d\zeta\nonumber\\
&=& I_1|_{w=z}+I_2\nonumber
\end{eqnarray}
We have
$$I_1=\frac{1}{l!}\frac{d^l}{dw^l}\left\{\int_{\triangle} \frac{(\bar\zeta-\bar z)^l}{\zeta-w}d\bar\zeta\wedge d\zeta-\int_{\epsilon(z)} \frac{(\bar\zeta-\bar z)^l}{\zeta-w}d\bar\zeta\wedge d\zeta\right\}$$
where $w\in \epsilon(z)$. By Lemma 3.5, $I_1=0$. $\lim_{\epsilon\to 0}I_2=0$ is obvious. This proof is complete.
\end{proof}
Now we are ready to give proof of Therem 3.4.
\begin{proof} It suffices to show that if $f\in C^{k+\alpha}(D)$, then $^{k+1}Tf$ has a total differential, which means we must show that there are numbers $A$ and $B$, depending on $z$, such that
$$^{k+1}Tf(z)-^{k+1}Tf(z')=A(z-z')+B(\bar z-\bar z')+\varepsilon(z,z')\cdot|z-z'|$$
where $\varepsilon (z,z')\to 0$ as $z'\to z$. It is claimed that $A=^{k+2}Tf(z)$ and $B=\partial^k f(z)$. To this end, denote the left side in the above by $I$, and restrict $z'$ so that $\rho=|z-z'|<\frac{1}{2}(R-|z|)$. Let $\delta$ be the disk of radius $\rho$ and center $(z+z')/2$. Notice $\delta\subset D$. Then
$$I=^{k+1}T_\delta f(z)- ^{k+1}T_\delta f(z')-\frac{k!}{2\pi i} \int_{D\setminus\delta}\bigg\{\frac{f(\zeta)-P_{k-1}(\zeta,z)}{(\zeta-z)^{k+1}}-\frac{f(\zeta)-P_{k-1}(\zeta,z')}{(\zeta-z')^{k+1}}\bigg\}d\bar\zeta\wedge d\zeta.$$
We write by Lemma 2.2 that
$$f(\zeta)=P_{k-1}(\zeta, z)+\frac{1}{k!}\sum_{i+j=k}\partial^i\bar\partial^j f(z)(\zeta-z)^i(\bar\zeta-\bar z)^j+E_{k}(\zeta, z)$$
where
$$|E_{k}(\zeta, z)|\leq
\frac{1}{k!}\bigg\{\sum_{i+j=k}H_\alpha[\partial^i\bar\partial^j f]\bigg\}|\zeta-z|^{k+\alpha}.$$

We now consider the terms in order. First we need the following, using Lemma 3.6 and Corollary 3.5,
\begin{eqnarray}
^{k+1}T_\delta f(z)&=&\frac{-k!}{2\pi i}\int_\delta\frac{f(\zeta)-P_{k-1}(\zeta,z)}{(\zeta-z)^{k+1}}d\bar\zeta\wedge d\zeta\nonumber\\
&=&\frac{-1}{2\pi i }\sum_{i+j=k}\partial^i\bar\partial^j f(z)\int_\delta\frac{(\zeta-z)^i(\bar\zeta-\bar z)^j}{(\zeta-z)^{k+1}}d\bar\zeta\wedge d\zeta\nonumber\\
&+&\frac{-k!}{2\pi i}\int_\delta\frac{E_{k}(\zeta,z)}{(\zeta-z)^{k+1}}d\bar\zeta\wedge d\zeta\nonumber\\
&=&\partial^k f(z)\frac{-1}{2\pi i}\int_\delta\frac{1}{(\zeta-z)}d\bar\zeta\wedge d\zeta+\frac{-k}{2\pi i}\int_\delta\frac{E_{k}(\zeta,z)}{(\zeta-z)^{k+1}}d\bar\zeta\wedge d\zeta\nonumber\\
&=&\partial^k f(z)(\bar z-\frac{\bar z+\bar z'}{2})+\frac{-k}{2\pi i}\int_\delta\frac{E_{k}(\zeta,z)}{(\zeta-z)^{k+1}}d\bar\zeta\wedge d\zeta\nonumber\\
&=&\frac{1}{2}\partial^k f(z)(\bar z-\bar z')+I_3.\nonumber
\end{eqnarray}
Similarly we have
\begin{eqnarray}
^{k+1}T_\delta f(z')&=&\frac{1}{2}\partial^k f(z')(\bar z'-\bar z)+\frac{-k}{2\pi i}\int_\delta\frac{E_{k}(\zeta,z')}{(\zeta-z')^{k+1}}d\bar\zeta\wedge d\zeta\nonumber\\
&=&\frac{1}{2}\partial^k f(z')(\bar z'-\bar z)+I_3'.\nonumber
\end{eqnarray}
Now we estimate $I_3, I_3'$. It follows
$$|I_3|=\bigg|\frac{-k}{2\pi i}\int_\delta\frac{E_{k}(\zeta,z)}{(\zeta-z)^{k+1}}d\bar\zeta\wedge d\zeta\bigg|\leq C\sum_{i+j=k}H_\alpha[\partial^i\bar\partial^j f]\int_0^{2\pi}\int_0^{r_1(\theta)}\frac{r^{k+\alpha}2rdrd\theta}{r^{k+1}}$$
where $r_1(\theta)$ is the distance from $z$ to the boundary of $\delta$ in direction $\theta$; note $r_1(\theta)\leq2\rho$. Continuing the computation, we have
$$|I_3|\leq C\sum_{i+j=k}H_\alpha[\partial^i\bar\partial^j f](2\rho)^{1+\alpha}$$
which approaches $0$ as $\rho\to 0$. The same estimate holds true for $I_3'$. Thus we have
\begin{eqnarray}
&&|^{k+1}T_\delta f(z)- ^{k+1}T_\delta f(z')-\partial^kf(z)(\bar z-\bar z')|\nonumber\\
&\leq& |I_3|+|I_3'|+\bigg|\frac{1}{2}\partial^k f(z)(\bar z-\bar z')-\frac{1}{2}\partial^k f(z')(\bar z'-\bar z)-\partial^kf(z)(\bar z-\bar z')\bigg|\nonumber\\
&\leq& |I_3|+|I_3'|+\frac{1}{2}H_\alpha[\partial^k f]|z-z'|^{1+\alpha}.\nonumber
\end{eqnarray}
On the other hand, to continue to estimate, we have
\begin{eqnarray}
&&\int_{D\setminus\delta}\bigg\{\frac{f(\zeta)-P_{k-1}(\zeta,z)}{(\zeta-z)^{k+1}}-\frac{f(\zeta)-P_{k-1}(\zeta,z')}{(\zeta-z')^{k+1}}\bigg\}d\bar\zeta\wedge d\zeta \nonumber\\
&=&\int_{D\setminus\delta}f(\zeta)\bigg(\frac{1}{(\zeta-z)^{k+1}}-\frac{1}{(\zeta-z')^{k+1}}\bigg)d\bar\zeta\wedge d\zeta-
\int_{D\setminus\delta}\frac{P_{k-1}(\zeta, z)}{(\zeta-z)^{k+1}}d\bar\zeta\wedge d\zeta \nonumber\\
&+&\int_{D\setminus\delta}\frac{P_{k-1}(\zeta, z')}{(\zeta-z')^{k+1}}d\bar\zeta\wedge d\zeta\nonumber\\
&=& I_4+I_5+I_6.\nonumber
\end{eqnarray}
We need the following lemma to proceed.
\begin{lem}
If $m\geq 2, n\geq 0$, then
$$\int_{D\setminus\delta}\frac{(\bar\zeta-\bar z)^n}{(\zeta-z)^m}d\bar\zeta\wedge d\zeta=\int_{D\setminus\delta}\frac{(\bar\zeta-\bar z')^n}{(\zeta-z')^m}d\bar\zeta\wedge d\zeta=0.$$
Actually more is true:
$$\int_{D\setminus\delta}\frac{(\bar\zeta-\bar s)^n}{(\zeta-s)^m}d\bar\zeta\wedge d\zeta=0$$
for any $s\in\delta$.
\end{lem}
\begin{proof} We have
$$\int_{D\setminus\delta}\frac{(\bar\zeta-\bar z)^n}{(\zeta-z)^m}d\bar\zeta\wedge d\zeta=\frac{1}{(m-1)!}\frac{d^{m-1}}{dw^{m-1}}
\int_{D\setminus\delta}\frac{(\bar\zeta-\bar z)^n}{\zeta-w}d\bar\zeta\wedge d\zeta\bigg|_{w=z},$$
$$=\frac{1}{(m-1)!}\frac{d^{m-1}}{dw^{m-1}}\{
\int_{D}\frac{(\bar\zeta-\bar z)^n}{\zeta-w}d\bar\zeta\wedge d\zeta-\int_{\delta}\frac{(\bar\zeta-\bar z)^n}{\zeta-w}d\bar\zeta\wedge d\zeta\}\bigg|_{w=z},$$
which is zero by applying Lemma 3.4.
\end{proof}
By Lemma 3.8, we have
$$I_5=I_6=0.$$
Now we only have to estimate $I_4$. Indeed,
$$I_4=-(k+1)\int_{z'}^z\bigg(\int_{D\setminus\delta}\frac{f(\zeta)}{(\zeta-w)^{k+2}}d\bar\zeta\wedge d\zeta\bigg) dw$$
where the integration with respect to $w$ is along the straight line segment from $z'$ to $z$. Now it suffices to estimate
\begin{eqnarray}
&&\int_{D\setminus\delta}\frac{f(\zeta)}{(\zeta-w)^{k+2}}d\bar\zeta\wedge d\zeta-\int_{D}\frac{f(\zeta)-P_k(\zeta, z)}{(\zeta-z)^{k+2}}d\bar\zeta\wedge d\zeta\nonumber\\
&=&\int_{D\setminus\delta}\frac{f(\zeta)}{(\zeta-w)^{k+2}}d\bar\zeta\wedge d\zeta-
\int_{D\setminus\delta}\frac{f(\zeta)-P_k(\zeta, z)}{(\zeta-z)^{k+2}}d\bar\zeta\wedge d\zeta\nonumber\\
&-&\int_{\delta}\frac{f(\zeta)-P_k(\zeta, z)}{(\zeta-z)^{k+2}}d\bar\zeta\wedge d\zeta\nonumber\\
&=&\int_{D\setminus\delta}f(\zeta)\bigg\{\frac{1}{(\zeta-w)^{k+2}}-\frac{1}{(\zeta-z)^{k+2}}\bigg\}d\bar\zeta\wedge d\zeta-\int_{\delta}\frac{f(\zeta)-P_k(\zeta, z)}{(\zeta-z)^{k+2}}d\bar\zeta\wedge d\zeta\nonumber\\
&=&I_7+I_8.\nonumber
\end{eqnarray}
Here we have used Lemma 3.8 to conclude
$$\int_{D\setminus\delta}\frac{P_k(\zeta, z)}{(\zeta-z)^{k+2}}d\bar\zeta\wedge d\zeta=0.$$
Now we have
\begin{eqnarray}
I_7&=&\int_{D\setminus\delta}f(\zeta)\bigg\{\frac{1}{(\zeta-w)^{k+2}}-\frac{1}{(\zeta-z)^{k+2}}\bigg\}d\bar\zeta\wedge d\zeta\nonumber\\
&=&(k+2)\int_{D\setminus\delta}f(\zeta)\int_w^z\frac{ds}{(\zeta-s)^{k+3}}d\bar\zeta\wedge d\zeta\nonumber\\
&=&(k+2)\int_w^z ds \int_{D\setminus\delta}\frac{f(\zeta)-P_{k}(\zeta, s)}{(\zeta-s)^{k+3}}d\bar\zeta\wedge d\zeta\nonumber
\end{eqnarray}
where we have used that
$$\int_{D\setminus\delta}\frac{P_{k}(\zeta, s)}{(\zeta-s)^{k+3}}d\bar\zeta\wedge d\zeta=0,$$
for $s\in\delta$ by Lemma 3.8. Now
\begin{eqnarray}
|I_7|&\leq& C|w-z|\bigg |\int_{D\setminus\delta}\frac{f(\zeta)-P_{k}(\zeta, s)}{(\zeta-s)^{k+3}}d\bar\zeta\wedge d\zeta\bigg|\nonumber\\
&\leq& C|w-z|\sum_{i+j=k}H_\alpha[\partial^i\bar\partial^j f]2\pi\int_{\rho/2}^{2R}\frac{r^{k+\alpha}2rdr}{r^{k+3}}\nonumber\\
&=&C|w-z|\sum_{i+j=k}H_\alpha[\partial^i\bar\partial^j f]2\pi \frac{1}{1-\alpha}((\rho/2)^{\alpha-1}-(2R)^{\alpha-1})\nonumber\\
&\leq& C\sum_{i+j=k}H_\alpha[\partial^i\bar\partial^j f]\rho^\alpha.\nonumber
\end{eqnarray}
In the last inequality we have used that $|w-z|\leq |z-z'|=\rho.$
We have thus shown that
$$^{k+1}Tf(z)-^{k+1}Tf(z')=^{k+2}Tf(z)(z-z')+\partial^k f(z)(\bar z-\bar z')+O(|z-z'|^{1+\alpha})$$
as $z'\to z$. Thus $^{k+1}Tf$ has a total differential and $\partial (^{k+1}Tf)=^{k+2}Tf$.
\end{proof}
In what follows, we will show that $^{k+2}Tf\in C^\alpha(D)$ if $f\in C^{k+\alpha}(D)$.
\begin{lem}
Let $f\in C^{k+\alpha}(D)$. Then
$$|^{k+2}Tf|\leq CR^\alpha \sum_{i+j=k}H_\alpha[\partial^i\bar\partial^jf].$$

\end{lem}

\begin{proof}
We compute, using Lemma 2.2,
\begin{eqnarray}
|^{k+2}Tf(z)|&=&\frac{(k+1)!}{2\pi}\bigg|\int_D\frac{f(\zeta)-P_{k}(\zeta,z)}{(\zeta-z)^{k+2}}d\bar\zeta\wedge d\zeta\bigg|,\nonumber\\
&\leq&C\sum_{i+j=k}H_\alpha[\partial^i\bar\partial^jf]\int_0^{2\pi}\int_0^{2R}\frac{r^{k+\alpha}2rdrd\theta}{r^{k+2}}\nonumber\\
&\leq& CR^\alpha\sum_{i+j=k}H_\alpha[\partial^i\bar\partial^jf].\nonumber
\end{eqnarray}
\end{proof}
The following is well known [NW, K]
\begin{lem}
If $f\in C^{1+\alpha}(D)$, then
$$^2Tf=T(\partial f)- S_b(f).$$
\end{lem}
\begin{lem}
If $f\in C^{k+\alpha}(D)$, then for $k\geq 1$
$$^{k+2}Tf=^2T(\partial^k f)-\sum_{i=1}^k\partial^i S_b(\partial^{k-i}f).$$
\end{lem}
\begin{proof}
By Lemma 3.10, we have
$$^2Tf=T(\partial f)- S_b(f).$$
Taking derivative, we have
$$\partial ^2Tf=\partial T(\partial f)- \partial S_b(f),$$
which is equivalent to
$$ ^3Tf=^2T(\partial f)- \partial S_b(f).$$
A mathematical induction finishes the proof.
\end{proof}
\begin{lem}
If $f\in C^{k+\alpha}(D)$, then
$$H_\alpha[\partial^k S_b f]\leq  C\sum_{i+j=k}H_\alpha[\partial^i\bar\partial^j f].$$

\end{lem}
\begin{proof}
We notice that
$$\partial^k S_b f(z)=\frac{k!}{2\pi i}\int_C\frac{f(\zeta)}{(\zeta-z)^{k+1}}d\bar\zeta,$$
and $S_b f(z)$ is holomorphic in $\mathrm{Int}(D)$.
First we need the following
$$\int_C\frac{P_k(\zeta, w)}{(\zeta-w)^{k+2}}d\bar\zeta=0,$$
for $w\in \mathrm{Int}(D)$, and where $P_k(\zeta, w)$ is the Taylor expansion of $f(\zeta)$ of order $k$ at $w$. Indeed,
$$\int_C\frac{P_k(\zeta, w)}{(\zeta-w)^{k+2}}d\bar\zeta=\sum_{l=0}^k\frac{1}{l!}\sum_{i+j=l}\partial^i\bar\partial^j f(w)\int_C\frac{(\zeta-w)^i(\bar\zeta-\bar w)^j}{(\zeta-w)^{k+2}}d\bar\zeta.$$
It suffices to prove the following
$$\int_C\frac{(\bar\zeta-\bar w)^n}{(\zeta-w)^{m+2}}d\bar\zeta=0$$
for $m,n\geq 0$. This is equivalent to prove that
$$\int_C\frac{\bar\zeta^n}{(\zeta-w)^{m+2}}d\bar\zeta=0$$
for $m,n\geq 0$. This is equivalent to prove that
$$\int_C\frac{1}{\zeta^{n+2}(\zeta-w)^{m+2}}d\zeta=0,$$
which is true by the residue theorem. Now we are ready to estimate.
\begin{eqnarray}
&&\frac{2\pi i}{k!}(\partial^k S_b f(z)-\partial^k S_b f(z'))\nonumber\\ &=&\int_Cf(\zeta)\bigg[\frac{1}{(\zeta-z)^{k+1}}-\frac{1}{(\zeta-z')^{k+1}}\bigg]d\bar\zeta\nonumber\\
&=&\frac{1}{k+1}\int_Cf(\zeta)\int_z^{z'}\frac{dw}{(\zeta-w)^{k+2}}d\bar\zeta\nonumber\\
&=&\frac{1}{k+1}\int_z^{z'}dw\int_C \frac{f(\zeta)}{(\zeta-w)^{k+2}}d\bar\zeta\nonumber\\
&=&\frac{1}{k+1}\int_z^{z'}dw\int_C \frac{f(\zeta)-P_k(\zeta,w)}{(\zeta-w)^{k+2}}d\bar\zeta.\nonumber
\end{eqnarray}
In order to apply estimates in [NW], we convert integral to be on unit circle. Setting $w=R\omega$ and $\zeta=R\eta$, this becomes
$$\frac{1}{(k+1)R^k}\int_{z/R}^{z'/R}d\omega\int_{|\eta|=1}\frac{f(R\eta)-P_k(R\eta,R\omega)}{(\eta-\omega)^{k+2}}d\bar\eta$$
where the line integral is taken on the shorter segment of the circle through $z/R$ and $z'/R$ and orthogonal to the unit circle(see [NW] (6.2a)).
Further, setting $\tau=(\eta-\omega)/(1-\bar\omega\eta)$, we have
$$\eta=\frac{\tau+\omega}{1+\bar\omega\tau}, \mbox{  }\eta-\omega=\frac{\tau(1-\omega\bar\omega)}{1+\tau\bar\omega},\mbox{  } d\eta=\frac{(1-\omega\bar\omega)}{(1+\tau\bar\omega)^2}d\tau.$$
Hence
\begin{eqnarray}
&\frac{2\pi }{k!}&|\partial^k S_b f(z)-\partial^k S_b f(z') |\nonumber\\
&\leq &\frac{C}{(R^k}\int_{z/R}^{z'/R}|d\omega|\int_{|\eta|=1}\bigg|\frac{f(R\eta)-P_k(R\eta,R\omega)}{(\eta-\omega)^{k+2}}\bigg||d\bar\eta|\nonumber\\
&\leq& \frac{C}{R^k}\int_{z/R}^{z'/R}|d\omega|\int_{|\eta|=1}\frac{|f(R\frac{\tau+\omega}{1+\bar\omega\tau})-P_k(R\frac{\tau+\omega}{1+\bar\omega\tau},R\omega)||1+\tau\bar\omega|^k}{(1-\omega\bar\omega)^{k+1}}|d\eta|\nonumber\\
&\leq& \frac{C}{R^k}\sum_{i+j=k}H_\alpha[\partial^i\bar\partial^j f]\int_{z/R}^{z'/R}|d\omega|\int_{|\eta|=1}\frac{|R\frac{\tau+\omega}{1+\bar\omega\tau}-R\omega|^{k+\alpha}
|1+\tau\bar\omega|^k}{(1-\omega\bar\omega)^{k+1}}|d\eta|\nonumber\\
&\leq& CR^\alpha\sum_{i+j=k}H_\alpha[\partial^i\bar\partial^j f]\int_{z/R}^{z'/R}\frac{|d\omega|}{(1-\omega\bar\omega)^{1-\alpha}}\int_{|\tau|=1}\frac{|d\tau|}{|1+\bar\omega\tau|^\alpha}.\nonumber
\end{eqnarray}
According to [NW](6.2a and p473), one has
$$\int_{|\tau|=1}\frac{|d\tau|}{|1+\bar\omega\tau|^\alpha}\leq \frac{4\pi}{1-\alpha},$$
$$\int_{z/R}^{z'/R}\frac{|d\omega|}{(1-\omega\bar\omega)^{1-\alpha}}\leq \frac{2}{\alpha}|z-z'|^\alpha\frac{1}{R^\alpha}.$$
Hence
$$H_\alpha[\partial^k S_bf]\leq C\sum_{i+j=k}H_\alpha[\partial^i\bar\partial^j f].$$
\end{proof}
The following is the key result to be used to prove existence of solutions in what follows.
\begin{thm}
If $f\in C^{k+\alpha}(D)$, then
$$H_\alpha[^{k+2}Tf]\leq C\sum_{i+j=k}H_\alpha[\partial^i\bar\partial^j f],$$
$$H_\alpha[^{k+2}\overline Tf]\leq C\sum_{i+j=k}H_\alpha[\partial^i\bar\partial^j f].$$
\end{thm}
\begin{proof}
By Lemma 3.11, and applying Lemma 3.12, we have
\begin{eqnarray}
^{k+2}Tf&=&^2T(\partial^k f)-\sum_{i=1}^k\partial^i S_b(\partial^{k-i}f).\nonumber\\
H_{\alpha}[^{k+2}Tf]&\leq& H_\alpha[^2T(\partial^k f)]+\sum_{i=1}^kH_\alpha[\partial^i S_b(\partial^{k-i}f)]\nonumber\\
&\leq& C H_\alpha[\partial^k f]+C\left\{\sum_{i=1}^k \frac{1}{i}\sum_{p+q=i}H_\alpha[\partial^p\bar\partial^q\partial^{k-i} f]\right\}\nonumber\\
&\leq& C\sum_{i+j=k}H_\alpha[\partial^i\bar\partial^j f].\nonumber
\end{eqnarray}
\end{proof}

\begin{lem}
If $h\in C^{k+\alpha}(D)$ $(k\geq 0)$, then
$$\|Th\|^{(k+1)}\leq C\|h\|^{(k)},$$
$$\|\overline Th\|^{(k+1)}\leq C \|h\|^{(k)}.$$
\end{lem}
\begin{proof}
Let $i+j=k+1$. If $j\geq 1$, then $\partial^i\bar\partial^j Th=\partial^i\bar\partial^{j-1}h$.
Otherwise, $\partial^{k+1}Th=^{k+2}Th$. We have by Lemma 3.9, 3.13
\begin{eqnarray}
\|^{k+2}Th\|&=&|^{k+2}Th|+(2R)^\alpha H_\alpha[^{k+2}Th]\nonumber\\
&\leq&CR^\alpha \sum_{i+j=k}H_\alpha[\partial^i\bar\partial^jh]+(2R)^\alpha C\sum_{i+j=k}H_\alpha[\partial^i\bar\partial^j h]\nonumber\\
&\leq& (CR^\alpha +(2R)^\alpha C)\sum_{i+j=k}H_\alpha[\partial^i\bar\partial^j h]\nonumber\\
&\leq& C\|h\|^{(k)}.\nonumber
\end{eqnarray}
In the last inequality we have used that $H_\alpha[\partial^i\bar\partial^j h]\leq (2R)^{-\alpha}\|h\|^{(k)}$ for $i+j=k$.
The proof is now complete.
\end{proof}
The following is the main result of this section.
\begin{thm}
If $h\in C^{\alpha}(D)$ and $\mu+\nu=m$ , then
$$\|T^\nu\overline T^\mu h\|^{(m)}\leq C \|h\|.$$

\end{thm}
\begin{proof}
Since, if $\nu\geq 1$,
$$\|T^\nu\overline T^\mu h\|^{(m)}=\|T(T^{\nu-1}\overline T^\mu h)\|^{(m-1+1)},$$
which is less than, by invoking Lemma 3.14
$$C\|T^{\nu-1}\overline T^\mu h\|^{(m-1)},$$
which is less than, by repeating the argument,
$C\| h\|.$
This gives the proof of the theorem.
\end{proof}

\section{A High order integral operator }\label{sec1}
In order to prove the existence of solutions of system with arbitrary initial values we will first solve the equation with zero initial values at the origin. Namely, we are to solve the system of equations
\begin{eqnarray}
\partial^\mu\bar\partial^\nu u(z) &=& a(z, u(z), \mathcal{D}^1 u(z), ...,\mathcal{D}^{m-1} u(z))
\end{eqnarray}
with zero initial values at the origin:
\begin{eqnarray}
 \mathcal{D}^ku(0)&=& 0, \mbox{ for  } 0\leq k\leq m-1
\end{eqnarray}
Every solution $u$ of (1) satisfies the following integral equation, by Lemma 3.2,
\begin{eqnarray}
u(z)=\psi(z) +T^\nu\overline T^\mu a(\zeta, u, \mathcal{D}^1 u, ...,\mathcal{D}^{m-1} u)(z),
\end{eqnarray}
where  $\partial^\mu\bar\partial^\nu\psi(z)=0$.
In what follows, we will introduce function spaces that are Banach with norms to reflect the initial values condition, and
will modify the equation (3) to fit the defined Banach space and to apply fixed point theorem to produce the needed solutions. In doing so, we will take a full advantage of
two parameters: the radius $R$ of $D$ and the radius $\gamma$ of a closed ball in the Banach space where the solutions are sought from.
\subsection{Iteration procedures} For the sake of generality we consider a Banach space $\mathbf{B}(R)$ depending on a positive parameter $R$. $R$ will be less than some fixed positive quantity.
The norm on $\mathbf{B}(R)$ is denoted by $\|\cdot\cdot\cdot\|$; it may also depend on $R$.
The following deals with a map $\mathbf{\Theta}:\mathbf{A}(R,\gamma)\to \mathbf{B}(R)$, where $\mathbf{A}(R,\gamma)(\gamma>0)$ is a closed subset of $\mathbf{B}(R)$ defined as
$\mathbf{A}(R,\gamma)=\{\Omega|\|\Omega\|\leq \gamma\}$. Here we actually have a family of such maps with two parameters $R,\gamma$.
\begin{lem}
 Let $\mathbf{\Theta}:\mathbf{A}(R,\gamma)\to \mathbf{B}(R)$, and
and let $\delta(R,\gamma)$ and $\eta(R,\gamma)$ exist such that for all $\Omega,\Omega'\in \mathbf{A}(R,\gamma)$
$$\|\mathbf{\Theta}(\Omega')-\mathbf{\Theta}(\Omega)\|\leq \delta(R,\gamma)\|\Omega'-\Omega\|,$$
$$\|\mathbf{\Theta}(\Omega)\|\leq \eta(R,\gamma).$$
If there exist $R_0>0, \gamma_0>0$ such that  $\delta(R,\gamma_0)\leq 3/4$ and $\eta(R,\gamma_0)\leq \frac{\gamma_0}{2}$ for $R\leq R_0$. Let $\psi\in\mathbf{B}(R)$ be such that
$\|\psi\|\leq \frac{\gamma_0}{2}$. Then the following equation
$$ \Omega=\psi+\mathbf{\Theta}(\Omega)$$
has a unique solution in $\mathbf{A}(R,\gamma_0)$ for $R\leq R_0$. The solution $\Omega$ is the limit of the sequence
\begin{eqnarray}
\Omega_{N+1}=\psi+\mathbf{\Theta}(\Omega_N)\mbox{  for } N=1,2,...\nonumber
\end{eqnarray}
where $\Omega_1\in \mathbf{A}(R,\gamma_0)$.
\end{lem}

\subsection{Integral equations on Banach spaces}

Here we fix the Banach space to work on.
Let $\mathbf{B}(R)$ be $[C^{m+\alpha}_0(D)]^n$, which is, by Lemma 2.6 a Banach space with norm $||\cdot\cdot\cdot||^{(m)}$ defined as follows:
if $f=(f_1,..., f_n)\in \mathbf{B}(R)$, then $\|f\|^{(m)}=\max_{1\leq i\leq n}\|f_i\|^{(m)}$.
Let $\gamma$ be a positive number, and we consider
$$\mathbf{A}(R,\gamma)=\{ f\in \mathbf{B}(R)|\|f\|^{(m)}\leq \gamma\}.$$
We note that $\mathbf{A}(R,\gamma)$ is a closed subset of $\mathbf{B}(R)$.
Now we assume  $u=(u^1,...,u^n)$ and $a=(a^1,...,a^n)$. According to (1) , we consider equations
$$u^i=\psi^i+T^\nu\overline T^\mu a^i(z, u, \mathcal{D}^1 u, ...,\mathcal{D}^{m-1} u)$$
where $\psi=(\psi^1,...,\psi^n)$ is such that $\partial^\mu\bar\partial^\nu\partial \psi=0$.
To simplify the notation we define
$$\omega^i(f)=T^\nu\overline T^\mu a^i(z, f, \mathcal{D}^1 f, ...,\mathcal{D}^{m-1} f)$$
for all $f\in \mathbf{B}(R)$ and $i=1,2,...,n$.
To assure the zero initial satisfied, we consider
\begin{eqnarray}
\mathbf{\Theta}^i(f)(\zeta)&=&\omega^i(f)(\zeta)-\sum_{p=0}^{m-1}\frac{1}{p!}\sum_{k+l=p}[\partial^k\bar\partial^l \omega^i(f)](0)\zeta^k\bar\zeta^l.
\end{eqnarray}

We first notice that $T, \overline T $ map $C^\alpha(\D)\to C^{1+\alpha}(D)$, and it follows, by the construction (1), $\mathbf{\Theta}^i(f)(\zeta)\in
C^{m+\alpha}_0(D)$ for $f\in \mathbf{B}(R)$. Thus we define a map from $\mathbf{B}(R)$ to $\mathbf{B}(R)$.
\begin{eqnarray}
\mathbf{\Theta}&:&\mathbf{B}(R)\rightarrow\mathbf{B}(R)\nonumber\\
\mathbf{\Theta}(f)&=&(\mathbf{\Theta}^1(f),...,\mathbf{\Theta}^n(f)).
\end{eqnarray}
Here we recall again $\|\mathbf{\Theta}(f)\|^{(m)}=\max_{1\leq i\leq n}\|\mathbf{\Theta}^i(f)\|^{(m)}$ and recall that for $f\in C^{m+\alpha}_0(D)$,
$\|f\|^{(m)}=\max_{k+l=m}\{\|\partial^k\bar\partial^lf\|\}$. In order to apply iteration procedure of Lemma 4.1,
we will first estimate
$$\|\mathbf{\Theta}^i(f)-\mathbf{\Theta}^i(g)\|^{(m)}=\|\mathbf{\omega}^i(f)-\mathbf{\omega}^i(g)\|^{(m)}$$
for all $f, g\in \mathbf{A}(R,\gamma)$ in terms of $\|f-g\|^{(m)}$.
and next we will want to estimate $\|\mathbf{\Theta}^i(f)\|^{(m)}= \|\mathbf{\omega}^i(f)\|^{(m)}$ in terms of a constant.

\subsubsection{Estimate of $\|\mathbf{\omega}^i(f)-\mathbf{\omega}^i(g)\|^{(m)}$}
In this section we always have functions $f,g$ to be in $\mathbf{A}(R,\gamma)$.
We first begin with the estimate by Theorem 3.15:
\begin{eqnarray}
&&\|\mathbf{\omega}^i(f)-\mathbf{\omega}^i(g)\|^{(m)}\nonumber\\
&=&\|T^\nu\overline T^\mu (a^i(z, f, \mathcal{D}^1 f, ...,\mathcal{D}^{m-1} f)-a^i(z, g, \mathcal{D}^1 g, ...,\mathcal{D}^{m-1}g))\|^{(m)}\nonumber\\
&\leq& C\|a^i(z, f, \mathcal{D}^1 f, ...,\mathcal{D}^{m-1} f)-a^i(z, g, \mathcal{D}^1 g, ...,\mathcal{D}^{m-1}g)\|,
\end{eqnarray}

In order to estimate (5), we first estimate the norm $|\cdot\cdot\cdot|$.  In fact, we have
\begin{eqnarray}
&&a^i(\zeta, \mathcal{D}^0f(\zeta), \mathcal{D}^1 f(\zeta), ...,\mathcal{D}^{m-1} f(\zeta))-a^i(\zeta, \mathcal{D}^0g(\zeta), \mathcal{D}^1 g(\zeta), ...,\mathcal{D}^{m-1} g(\zeta))\nonumber\\
&=&\int_0^1\frac{d}{dt}a^i(\zeta, tf(\zeta)+(1-t)g(\zeta),..., t\mathcal{D}^{m-1} f(\zeta)+(1-t)\mathcal{D}^{m-1} g(\zeta))dt\nonumber\\
&=&\sum_{j=0}^n\sum_{p=0}^{m-1}\sum_{k+l=p}A^{p,j}_{k,l}\partial^k\bar\partial^l(f_j-g_j)+\bar A^{p,j}_{k,l}\overline{\partial^k\bar\partial^l(f_j-g_j)}
,
\end{eqnarray}
where
\begin{eqnarray}
A^{p,j}_{k,l}=\int_0^1\frac{\partial a^i}{\partial \eta_{p,j}^{k,l}}(\zeta, \mathcal{W}^0,...,\mathcal{W}^{m-1})dt,\mbox{   }
\bar A^{p,j}_{k,l}=\int_0^1\frac{\partial a^i}{\bar\partial \eta_{p,j}^{k,l}}(\zeta,  \mathcal{W}^0,...,\mathcal{W}^{m-1})dt
\end{eqnarray}
where we have used the shorten notations:
$$\mathcal{W}^{k}=t\mathcal{D}^k f(\zeta)+(1-t)\mathcal{D}^k g(\zeta)$$
for $k=0, 1, ..., m-1$, and $\eta_{p,j}^{k,l}$ is a variable in $\eta_p$.
Therefore, taking norm $\|\cdot\cdot\cdot\|$ on (6) we have, using Lemma 2.4,
\begin{eqnarray}
&&\|a^i(\zeta, \mathcal{D}^0f, \mathcal{D}^1 f, ...,\mathcal{D}^{m-1} f)-a^i(\zeta, \mathcal{D}^0g, \mathcal{D}^1 g, ...,\mathcal{D}^{m-1} g)\|\nonumber\\
&\leq&\sum_{j=0}^n\sum_{p=0}^{m-1}\sum_{k+l=p}\|A^{p,j}_{k,l}\partial^k\bar\partial^l(f_j-g_j)\|+\|\bar A^{p,j}_{k,l}\overline{\partial^k\bar\partial^l(f_j-g_j)}\|\nonumber\\
&\leq&\sum_{j=0}^n\sum_{p=0}^{m-1}\sum_{k+l=p}\|A^{p,j}_{k,l}\|\|\partial^k\bar\partial^l(f_j-g_j)\|+\|\bar A^{p,j}_{k,l}\|\|\overline{\partial^k\bar\partial^l(f_j-g_j)}\|\nonumber\\
&\leq&\sum_{j=0}^n\sum_{p=0}^{m-1}\sum_{k+l=p}\{\|A^{p,j}_{k,l}\|+\|\bar A^{p,j}_{k,l}\|\}\|{\partial^k\bar\partial^l(f-g)}\|\nonumber\\
&\leq&C\sum_{j=0}^n\sum_{p=0}^{m-1}\sum_{k+l=p}\{\|A^{p,j}_{k,l}\|+\|\bar A^{p,j}_{k,l}\|\}R^{m-p}\|f-g\|^{(m)}.
\end{eqnarray}
In order to estimate the quantities $\|A^{p,j}_{k,l}\|, \bar A^{p,j}_{k,l}\|$  in (8), we need to study the ranges of $\mathcal{W}^k, k=0, 1, ..., m-1$ for $f, g\in \mathbf{A}(R,\gamma)$. The following is what we need.
\begin{lem}
If $f, g\in \mathbf{A}(R,\gamma)$, then
\begin{eqnarray}
|\mathcal{W}^k|&\leq& CR^{m-k}\gamma,\mbox{   } k=0, 1, ..., m-1.\nonumber
\end{eqnarray}
\end{lem}
\begin{proof}
If $f$ is a function in $C_0^{m+\alpha}(D)$, then by Lemma 2.4
we have, if $i+j=k\leq m-1$,
$$\|\partial^i\bar\partial^j f\|\leq CR^{m-k}\|f\|^{(m)}\leq CR^{m-k}\gamma.$$
Particularly, it implies
$$|\partial^i\bar\partial^j f|\leq CR^{m-k}\gamma,$$
whence, by the definition of norm, we have
$$|\mathcal{W}^k|\leq CR^{m-k}\gamma$$
for $k=0, 1, ..., m-1$.
\end{proof}
To continue on the estimates, we need a lemma on Lipschitz properties of $C_0^{m+\alpha}(D)$.
\begin{lem}
Let $f\in C_0^{m+\alpha}(D)$. Then, for $\zeta', \zeta\in D$, and $i+j=l\leq m-1$, we have
$$|\partial^i\bar\partial^j f(\zeta')-\partial^i\bar\partial^j f(\zeta)|\leq C R^{m-l-1}\|f\|^{(m)}|\zeta'-\zeta|.$$
\end{lem}
\begin{proof} We have
\begin{eqnarray}
\partial^i\bar\partial^j f(\zeta')-\partial^i\bar\partial^j f(\zeta)&=&\int_0^1\frac{d}{dt}\{\partial^i\bar\partial^j f(t\zeta'+(1-t)\zeta)\}dt\nonumber\\
&=&\int_0^1\partial^{i+1}\bar\partial^j f(t\zeta'+(1-t)\zeta)(\zeta'-\zeta)+\partial^{i}\bar\partial^{j+1} f(t\zeta'+(1-t)\zeta)\overline {(\zeta'-\zeta)}dt,\nonumber
\end{eqnarray}
whence
\begin{eqnarray}
|\partial^i\bar\partial^j f(\zeta')-\partial^i\bar\partial^j f(\zeta)|&\leq& (|\partial^{i+1}\bar\partial^j f|+|\partial^{i}\bar\partial^{j+1} f|)|\zeta'-\zeta|\nonumber\\
&\leq& 2\|f\|^{(l+1)}|\zeta'-\zeta|\nonumber\\
&\leq& CR^{m-l-1}\|f\|^{(m)}\nonumber
\end{eqnarray}
where in the last inequality we have used Lemma 2.4.
\end{proof}
In the following we will use a compact set in $\Omega$ defined as follows:
$$E(R, \gamma)=D\times\Pi_{k=0}^{m-1}\{z\in\co^{n2^k}||z|\leq CR^{m-k}\gamma\}.$$
 Now we first define some constants, whose usefulness will be self-evident during proofs below and later sections.
\begin{eqnarray}
A(R, \gamma)&=&\max\bigg\{\bigg|\frac{\partial a^i}{\partial \eta_{p,j}^{k,l}}\bigg|_{E(R,\gamma)},\bigg|\frac{\partial a^i}{\bar\partial \eta_{p,j}^{k,l}}\bigg|_{E(R,\gamma)} \bigg |,\bigg|\frac{\partial a^i}{\partial\zeta}\bigg|_{E(R,\gamma)},\bigg|\frac{\partial a^i}{\partial\bar\zeta}\bigg|_{E(R,\gamma)}\bigg\}\nonumber\\
H_\alpha^A[R,\gamma]&=&\max\bigg\{H_\alpha\bigg [\frac{\partial a^i}{\partial \eta_{p,j}^{k,l}}\bigg]_{E(R,\gamma)}, H_\alpha\bigg [\frac{\partial a^i}{\partial \bar\eta_{p,j}^{k,l}}\bigg]_{E(R,\gamma)},\bigg|\frac{\partial a^i}{\partial\zeta}\bigg|_{E(R,\gamma)},\bigg|\frac{\partial a^i}{\partial\bar\zeta}\bigg|_{E(R,\gamma)}\bigg\}\nonumber
\end{eqnarray}
where max is taken over $k+l=p; 0\leq p\leq m-1; i,j=1,..n$.
 Let $h$ be a function defined on open set $\Omega$ in $\co^N$. The H\"older constant of $h$ is defined
$$H_\alpha[h]=\sup\bigg\{\frac{|h(u)-h(v)|}{|u-v|^\alpha}\bigg|u,v\in \Omega\bigg\}.$$
Now if we restrict $h$ to be on one variable: $g(\zeta)=h(u_1,\cdot\cdot\cdot,u_i,\zeta,u_{i+1},\cdot\cdot\cdot,u_N)$, then
$H_\alpha[g]\leq H_\alpha[h]$. This fact is used below.

It is obvious from (7) that
\begin{eqnarray}
&&\left |A^{p,j}_{k,l}\right|\leq A(R,\gamma)
,\mbox{  }\left |\bar A^{p,j}_{k,l}\right |\leq A(R,\gamma).
\end{eqnarray}
Now we estimate $H_\alpha [A^{p,j}_{k,l}]$ by inserting terms,
\begin{eqnarray}
&&A^{p,j}_{k,l}(\zeta')-A^{p,j}_{k,l}(\zeta)\nonumber\\
&=&\int_0^1\{\frac{\partial a^i}{\partial \eta_{p,j}^{k,l}}(\zeta', \mathcal{W}^0(\zeta'),...,\mathcal{W}^{m-2}(\zeta'),\mathcal{W}^{m-1}(\zeta') )-\frac{\partial a^i}{\partial \eta_{p,j}^{k,l}}(\zeta, \mathcal{W}^0(\zeta),...,\mathcal{W}^{m-2}(\zeta),\mathcal{W}^{m-1}(\zeta) )\}dt\nonumber\\
&=&\int_0^1\{\frac{\partial a^i}{\partial \eta_{p,j}^{k,l}}(\zeta', \mathcal{W}^0(\zeta'),...,\mathcal{W}^{m-2}(\zeta'),\mathcal{W}^{m-1}(\zeta') )-\frac{\partial a^i}{\partial \eta_{p,j}^{k,l}}(\zeta, \mathcal{W}^0(\zeta'),...,\mathcal{W}^{m-2}(\zeta'),\mathcal{W}^{m-1}(\zeta') )\}dt\nonumber\\
&+&\int_0^1\{\frac{\partial a^i}{\partial \eta_{p,j}^{k,l}}(\zeta, \mathcal{W}^0(\zeta'),...,\mathcal{W}^{m-2}(\zeta'),\mathcal{W}^{m-1}(\zeta') )-\frac{\partial a^i}{\partial \eta_{p,j}^{k,l}}(\zeta, \mathcal{W}^0(\zeta), \mathcal{W}^1(\zeta'),...,\mathcal{W}^{m-2}(\zeta'),\mathcal{W}^{m-1}(\zeta') )\}dt\nonumber\\
&+&\cdot\cdot\cdot\nonumber\\
 &+&\int_0^1\{\frac{\partial a^i}{\partial \eta_{p,j}^{k,l}}(\zeta, \mathcal{W}^0(\zeta),...,\mathcal{W}^{m-2}(\zeta),\mathcal{W}^{m-1}(\zeta') )-\frac{\partial a^i}{\partial \eta_{p,j}^{k,l}}(\zeta, \mathcal{W}^0(\zeta), \mathcal{W}^1(\zeta),...,\mathcal{W}^{m-2}(\zeta),\mathcal{W}^{m-1}(\zeta) )\}dt,\nonumber
 \end{eqnarray}
 whence, since $a$ is of $C^{1+\alpha}$ and by Lemma 4.3,
\begin{eqnarray}
 &&|A^{p,j}_{k,l}(\zeta')-A^{p,j}_{k,l}(\zeta)|\nonumber\\
 &\leq&H_\alpha^A[R,\gamma]\bigg\{|\zeta'-\zeta|^\alpha+\sum_{l=0}^{m-1}\sum_{ i+j=l}
 (|\partial^i\bar \partial^j f(\zeta')-\partial^i\bar \partial^j f(\zeta)|+|\partial^i\bar \partial^j g(\zeta')-\partial^i\bar \partial^j g(\zeta)|)^\alpha\bigg\}\nonumber\\
 &\leq&
 H_\alpha^A[R,\gamma]\bigg\{|\zeta'-\zeta|^\alpha+C\sum_{l=0}^{m-1}\sum_{ i+j=l}
 (R^{m-l-1}\gamma)^\alpha|\zeta'-\zeta|^\alpha \bigg\}\nonumber\\
 \end{eqnarray}
Therefore, from (8) we have
\begin{eqnarray}
H_\alpha [A^{p,j}_{k,l}] &\leq&H_\alpha^A[R,\gamma]C(R,\gamma),
\end{eqnarray}
where $C(R,\gamma)=1+C\sum_{l=0}^{m-1}(R^{m-l-1}\gamma)^\alpha$.
By (10), (11), we have
\begin{eqnarray}
\|A^{p,j}_{k,l}\|&=&|A^{p,j}_{k,l}|+(2R)^\alpha H_\alpha [A^{p,j}_{k,l}] \nonumber\\
&\leq& A(R,\gamma)+(2R)^\alpha H_\alpha^A[R,\gamma]C(R,\gamma)
\end{eqnarray}
Of course, $\|\bar A^{p,j}_{k,l}\|$ satisfies the same estimate. 
Putting (6),(8),(11), and (12) together we have thus shown that
\begin{eqnarray}
\|a^i(\zeta, \mathcal{D}^0f, \mathcal{D}^1 f, ...,\mathcal{D}^{m-1} f)-a^i(\zeta, \mathcal{D}^0g, \mathcal{D}^1 g, ...,\mathcal{D}^{m-1} g)\|&\leq& \delta(R,\gamma)\|f-g\|^{(m)}\nonumber\\
\|\mathbf{\omega}^i(f)-\mathbf{\omega}^i(g)\|^{(m)}&\leq& \delta(R,\gamma)|\|f-g\|^{(m)}
\end{eqnarray}
where 
\begin{eqnarray}
\delta(R,\gamma)=C(\sum_{p=0}^{m-1}R^{m-p})(A(R,\gamma)+(2R)^\alpha H_\alpha^A[R,\gamma]C(R,\gamma)).
\end{eqnarray}
\subsubsection{Estimate of $\|\mathbf{\Theta}^i(f)\|^{(m)}$}
 We begin with
\begin{eqnarray}
&&a^i(\zeta, \mathcal{D}^0f(\zeta), \mathcal{D}^1 f(\zeta), ...,\mathcal{D}^{m-1} f(\zeta))-a^i(0,...,0)\nonumber\\
&=&\int_0^1\frac{d}{dt}a^i(t\zeta, tf(\zeta),..., t\mathcal{D}^{m-1} f(\zeta))dt\nonumber\\
&=&C^i\zeta+\bar C^i\bar\zeta+\sum_{j=0}^n\sum_{p=0}^{m-1}\sum_{k+l=p}B^{p,j}_{k,l}\partial^k\bar\partial^l(f_j)+\bar B^{p,j}_{k,l}\overline{\partial^k\bar\partial^l(f_j)}\nonumber
\end{eqnarray}
where we have
$$B^i=\int_0^1\frac{\partial a^i}{\partial\zeta}(t\zeta, tf(\zeta),..., t\mathcal{D}^{m-1} f(\zeta))dt,\bar B^i=\int_0^1\frac{\partial a^i}{\partial\bar\zeta}(t\zeta, tf(\zeta),..., t\mathcal{D}^{m-1} f(\zeta))dt,$$
$$B^{p,j}_{k,l}=\int_0^1\frac{\partial a^i}{\partial \eta_{p,j}^{k,l}}(t\zeta, tf(\zeta),..., t\mathcal{D}^{m-1} f(\zeta))dt,\bar B^{p,j}_{k,l}=\int_0^1\frac{\partial a^i}{\bar\partial \eta_{p,j}^{k,l}}(t\zeta, tf(\zeta),..., t\mathcal{D}^{m-1} f(\zeta))dt,$$
It follows by Theorem 3.15
\begin{eqnarray}
\|\omega^i(f)\|^{(m)}&\leq& C\|a^i(f)\|.\nonumber
\end{eqnarray}
Following the above estimate, we have
\begin{eqnarray}
\|a^i(f)\|\leq C(|a^i(0)|+A(R,\gamma)R+\delta(R,\gamma)).
\end{eqnarray}

\section{Proof of Theorem 1.1}

\subsection{Case of zero initial values}
First, we note from (14) that for each $\gamma$, $\lim_{\R\to 0}\delta(R,\gamma)=0$.
Let $\eta(R,\gamma)=C|a(0)|+CA(R,\gamma)R+C\delta(R,\gamma)$ from (16).
Now we choose $\gamma_0$ so large that $\frac{\gamma_0}{4}>C|a(0)|$. Then we choose $R$ sufficiently small that
$CA(R,\gamma)R+C\delta(R,\gamma_0)\leq \frac{\gamma_0}{4}$.
As a result, we have
$$\eta(R,\gamma_0)\leq\frac{\gamma_0}{2}.$$
If necessary, we choose $R$ further so that
$$\delta(R,\gamma_0)<3/4.$$
Let $\psi(z)$ be a homogenous polynomial map of degree $m$ such that $\|\psi\|\leq \frac{\gamma_0}{2}$. Then
we can apply Lemma 4.1 to
$$u=\psi+\mathbf{\Theta}(u)$$
on $A(R,\gamma_0)$ to get the desired unique solution that vanish up to order $m-1$ at the origin. $\psi$ be arbitrary, we have obtained infintely many solutions.
\subsection{General initial values}
To get the general case, we consider a new system, $p(z)=(p^1,...,p^n)$ with $\partial^i\bar\partial^j p(0)=c_{i, \bar j}$.
$$\partial^\mu\bar\partial^\nu \tilde u=a(z, \tilde u+p(z), \mathcal{D}^1\tilde u+\mathcal{D}^1p(z),..., \mathcal{D}^{m-1}\tilde u+\mathcal{D}^{m-1}p(z)).$$
This can be written as
$$\partial^\mu\bar\partial^\nu \tilde u=b(z, \tilde u, \mathcal{D}^1\tilde u,..., \mathcal{D}^{m-1}\tilde u).$$
 We can solve $\tilde u$ as just proved so that $\tilde u$ vanishes up to order $m-1$ at the origin.
Then $u=\tilde u+p(z)$ solves the original equation with desired property of derivatives at the origin since $\partial^\mu\bar\partial^\nu \tilde u=
\partial^\mu\bar\partial^\nu u$. The proof is complete.

\bigskip
Department of Mathematical Sciences

Indiana University - Purdue University Fort Wayne

Fort Wayne, IN 46805-1499, USA.

pan@ipfw.edu

\normalsize
\begin{thebibliography}{LNW}
\bibitem[A]{} {\sc L. \ Ahlfors}, {\sl An extension of Schwarz's lemma},
  Tran.\ A.M.S.  {\bf43} (1938), 359-364.
\bibitem[O]{} {\sc R. \ Osserman}, {\sl On the ineqauality $\Delta u\geq f(u)$},
  Pacific.\ J. (4) (1957), 1641-1647.
\bibitem[L]{} {\sc H. \ Lewy}, {\sl An example of a smooth linear partial differential equation without solution},
  Ann.\ of Math. (6) {\bf66} (1957), 155--158.
\bibitem[M]{} {\sc S. \ Mizohata}, {\sl Solutions nulles et solutions non analytiques},  J. Math. Kyoto Univ. {\bf128} (1988), 243--256.
\bibitem[NW]{nw} {\sc A.\ Nijenhuis} and {\sc W.\ Woolf}, {\sl
Some integration problems in almost-complex and complex manifolds },
  Ann.\ of Math.\ (3) {\bf77} (1963), 426--489.
\bibitem[PY]{}{\sc Y. \ Pan and \sc Y. Yu}, {\sl "On Local Solutions of Second Order Quasi-linear Elliptic Systems with Arbitrary 1-Jet at a Point}, preprint
\bibitem[PZ]{}{\sc Y. \ Pan and \sc Y. Zhang}, {\sl A residue-type phenomenon and its applications to higher order nonlinear systems of Poisson type },  J. Math. Anal. Appl. 495 (2021), no. 2,  Article 124749, 30 pages

\bibitem[K]{}{\sc S. \ Krantz}, {\sl Function Theory of Several Complex Variables},  AMS Chelsea Publishing 2001

\bibitem[V]{} {\sc I.N. \ Vekua}, {\sl Generalized analytic functions}, Pergamon 1962

\end{thebibliography}
\end{document}